\documentclass[12pt,a4paper,draft]{amsart}
\usepackage[english]{babel}
\usepackage{a4wide,amssymb,xypic,ifthen,hyperref}
\usepackage[mathcal]{euscript}

\newcommand{\qee}{ \hfill\hspace{2pt}$\triangle$}
\newcommand{\marginnote}[1]{\ifthenelse{\isodd{\thepage}}{\normalmarginpar}
{\reversemarginpar}\marginpar{\fbox{\parbox{15mm}{\sloppy\footnotesize #1}}}}

\newtheorem{thm}{Theorem}[section]

\newtheorem{lemma}[thm]{Lemma}
\newtheorem{prop}[thm]{Proposition}
\newtheorem{defin}[thm]{Definition}

\theoremstyle{remark}
\newtheorem{rema}[thm]{Remark}
 \newenvironment{remark}{\begin{rema}}{\qee\end{rema}}
\newtheorem{exe}[thm]{Example}
 
\newcommand{\Oc}{\mathcal O}

\renewcommand{\L}{\mathcal L}
\newcommand{\ccR}{\mathcal R}

\newcommand{\D}{\mathcal{D}}

\newcommand{\R}{\mathbb R}
\newcommand{\C}{\mathbb C}
\newcommand{\Q}{\mathbb Q}

\def\ker{\operatorname{ker}}

\def\rk{\operatorname{rk}}
\def\Ad{\operatorname{Ad}}

\def\End{\operatorname{End}}
\def\dim{\operatorname{dim}}

\def\Id{\operatorname{Id}}
\def\vol{\operatorname{vol}}

\newcommand{\fE}{{\mathfrak E}}

\newcommand{\g}{{\mathfrak g}}
\newcommand{\z}{{\mathfrak z}}

\begin{document}
\begin{flushright} SISSA preprint 03/2014/mate
\end{flushright} \bigskip
\title[Approximate HYM structures on principal Higgs bundles]{Approximate Hermitian-Yang-Mills  structures \\[8pt] on semistable principal Higgs bundles}
\bigskip
\date{Revised 29 June 2014}
\thanks{This research was partly supported by
GNSAGA-Istituto Nazionale per l'Alta Matematica  and PRIN ``Geometria delle variet\`a algebriche.''}

 \maketitle \thispagestyle{empty} \vspace{-3mm}
\begin{center}{\sc Ugo Bruzzo} \\
Scuola Internazionale Superiore di Studi Avanzati,\\ Via Bonomea 265, 34136
Trieste, Italia; \\ Istituto Nazionale di Fisica Nucleare, Sezione di Trieste \\ E-mail: {\tt bruzzo@sissa.it}
\\[6pt]
{\sc Beatriz Gra\~na Otero} \\
Departamento de Matem\'aticas, Pontificia Universidad Javeriana,
\\ Cra. 7$^{\hbox{\tiny \rm ma}}$ N$^{\mbox{\tiny\rm o}}$ 40-62, Bogot\'a, Colombia \\ E-mail: {\tt bgrana@javeriana.edu.co}

\end{center}

\vfill

\begin{abstract} We generalize the Hitchin-Kobayashi correspondence between semistability and the existence of approximate Hermitian-Yang-Mills structures to the case of principal Higgs bundles. We prove that a principal Higgs bundle $\fE$  on  a compact K\"ahler manifold, with structure group a connected linear algebraic  reductive group $G$, is semistable  if and only if it admits an approximate Hermitian-Yang-Mills structure.
\end{abstract}

\newpage \tableofcontents

\section{Introduction}
The notion of slope stability introduced by Mumford and Takemoto \cite{MF,Take} allows one to construct good moduli spaces of vector bundles (or, more generally, of coherent sheaves) on projective varieties. These moduli spaces usually are separated schemes of finite type, and, if suitable conditions are satisfied, they are fine, e.g., they carry universal objects on them. The definition of slope stability makes sense also on a compact K\"ahler manifold, and one can construct corresponding moduli spaces also in this case. 
However, for bundles on a compact K\"ahler manifold, there is also a differential-geometric analogue of slope stability, namely, the existence of a bundle metric which satisfies a certain condition, called the Hermitian-Yang-Mills  condition: the mean curvature of the Chern connection singled out by the metric is a (constant) multiple of the identity endomorphism \cite{koba,Lub83,Don85,Don87}. Actually, it turns out that, for an irreducible bundle, stability is equivalent to the existence of an Hermitian-Yang-Mills metric: this is the celebrated Hitchin-Kobayashi correspondence. A similar situation prevails in the presence of a Higgs field, for some other kinds of decorated bundles  \cite{S88,S92, Ban, Sch},  in the case of principal bundles \cite{RamaSubra} and principal Higgs bundles \cite{AnBis,BisSchu}.

The weaker property of semistability is equivalent to the existence of an Hermitian-Yang-Mills  metric in an approximate sense. Let us clarify the relation between (semi)stability for principal or vector bundles on compact K\"ahler manifolds, with or without Higgs fields, and the existence on these bundles of (approximate) Hermitian-Yang-Mills structures. After realizing that the sheaf
of sections of a holomorphic Hermitian (ordinary) vector bundle
satisfying the Hermitian-Yang-Mills condition is polystable (i.e.,
it is a direct sum of stable sheaves having the same slope)
\cite{koba,Lub83}, a converse result was proved. First Donaldson
in the projective case \cite{Don85,Don87}, and then Uhlenbeck and Yau
in the compact K\"ahler case \cite{UY}, showed that a stable bundle admits
a (unique up to homotheties) Hermitian metric which satisfies the
Hermitian-Yang-Mills condition. One can also show that if an
Hermitian bundle satisfies the Hermitian-Yang-Mills condition in
an approximate sense, then it is semistable, while the converse was proved in \cite{koba} for $X$ projective and in 
\cite{Jacob12} in the compact K\"ahler case.

Later Simpson \cite{S88,S92} proved a Hitchin-Kobayashi
correspondence for Higgs vector bundles. Given a Higgs bundle equipped
with an Hermitian metric, one defines a natural connection (that
we call the \emph{Hitchin-Simpson connection}); when the latter satisfies the
Hermitian-Yang-Mills condition, then the Higgs bundle is
polystable, and \emph{vice versa}.
Our main aim is to provide a suitable definition of approximate Hermitian-Yang-Mills structure for principal Higgs bundles and show that whenever a
Higgs principal bundle satisfies an approximate Hermitian-Yang-Mills
condition, it is semistable, and {\em vice versa.}

In this paper, after reviewing some basic definitions and properties, we give appropriate definitions of the notions of semistability and approximate Hermitian-Yang-Mills structure for principal Higgs bundles. In particular, one says that a principal Higgs $G$-bundle admits an approximate Hermitian-Yang-Mills structure if for every value of a positive real parameter $\xi$, there is a reduction $\sigma_{\xi}$ of the structure group to a  maximal compact subgroup such that the mean curvature of the Hitchin-Simpson connection given by the reduction approximates a central element in the Lie algebra of $G$ up to $\xi$. Then one shows that a principal Higgs $G$-bundle admits an approximate Hermitian-Yang-Mills structure if and only if it is semistable. 
A sketch of the proof of this result is as follows. If the principal Higgs bundle $\fE=(E, \phi)$ admits an approximate Hermitian-Yang-Mills structure, one shows that the same holds for the adjoint Higgs bundle $\Ad(\fE)$; the latter then is semistable \cite{BG2}, so that $\fE$ is semistable as well. The proof of the converse result relies on an analysis of the flow of the Donaldson functional defined on the space of Hermitian metrics on the Higgs   bundle $\Ad(\fE)$, showing that the flow preserves the condition that an Hermitian metric on $\Ad(\fE)$ comes from a reduction of the structure group of $\fE$ to a maximal compact subgroup $K$ (this implements in the case of principal Higgs bundles the ideas used in \cite{BisJacStem} for principal bundles, but we provide here a more detailed description of this technique).

Some related work has been done in \cite{BBJS}, where the existence --- in a suitable sense --- of a limit Hermitian-Yang-Mills connection on unstable Higgs principal $G$-bundles is established.

\bigskip\section{Preliminaries}\label{Sec:2}
We give  the basic definitions concerning Higgs bundles. Let
$X$ be an $n$-dimensional compact \ K\"ahler manifold, with K\"ahler
form $\omega$. 
Unless otherwise stated, $X$ will be assumed to be connected. 

Given a coherent sheaf $F$ on $X$, we denote by
$\deg(F)$ its degree $$\deg(F) = \int_X c_1(F)\cdot \omega^{n-1}$$
and if $r=\rk(F)>0$ we introduce its \emph{slope}
$$\mu(F) = \frac{\deg(F)}{r}.$$

\begin{defin} A Higgs sheaf $\fE$ on $X$ is a coherent sheaf $E$ on $X$
endowed with a morphism $\phi \colon E \to E \otimes \Omega_X^1$ of
$\Oc_X$-modules such that the morphism $\phi\wedge\phi : E \to E
\otimes \Omega^2_X$ vanishes, where $\Omega_X^i$ is the sheaf of holomorphic $i$-forms
on  $X$. A Higgs subsheaf $F$ of a Higgs sheaf $\fE=(E,\phi)$
is a subsheaf of $E$ such that $\phi(F)\subset F\otimes\Omega_X^1$.
A Higgs bundle is a Higgs sheaf $\fE $ such that $E$ is a
locally-free $\Oc_X$-module.
\end{defin}

There exists a stability condition for Higgs sheaves, analogous to
that for ordinary sheaves, which makes reference only to
$\phi$-invariant subsheaves.

\begin{defin}  A Higgs sheaf $\fE=(E,\phi) $ on $X$ is semistable
(resp.~stable) if $E$ is torsion-free, and $\mu(F)\le \mu(E)$
(resp. $\mu(F)< \mu(E)$) for every proper nontrivial Higgs
subsheaf $F$ of $\fE$. Finally, $\fE$ is polystable if it is a direct
sum of stable Higgs sheaves with the same slope.
\end{defin}

\subsection{Connections and fibre metrics}\label{Sec:0}
Let $\fE = (E, \phi)$ be a Higgs bundle over a compact  K\"ahler manifold
$X$ equipped with an Hermitian  fibre metric $h$. Then there is on
$E$ a unique connection $D_{(E,h)}$ which is compatible with both
the metric $h$ and the holomorphic structure of $E$ \cite{Ati}. This is often
called the \emph{Chern connection} of the Hermitian bundle
$(E,h)$.

Now let  $\bar\phi$ be the adjoint of the morphism $\phi$ with
respect to the metric $h$, i.e., the morphism $\bar\phi: E
\rightarrow \Omega^{0,1}_X\otimes E $ such that $$h(s, \phi(t)) =
h(\bar\phi(s), t)$$ for all sections $s,t$ of $E$. It is easy to
check that the operator
\begin{equation}\label{e:simpsonconnection}
\D_{(\fE,h)} = D_{(E,h)} + \phi + \bar\phi
\end{equation}
defines a connection on the bundle $E$. One should notice that
this connection is neither compatible with the holomorphic
structure of $E$, nor with the Hermitian metric $h$.
\begin{defin} \label{simpsonconnection}
The connection \eqref{e:simpsonconnection} is called the
\emph{Hitchin-Simpson connection} of the Hermitian Higgs bundle $(\fE,h)$.
Its curvature will be denoted by $\ccR_{(\fE,h)} = \D_{(\fE,h)}
\circ \D_{(\fE,h)}$. 
\end{defin}
If $\fE$ is a  Higgs line bundle the notion of Hermitian flatness coincides
with the usual one.

We shall denote by
$\mathcal{K}_{(\fE,h)}\in\End(E)$ the mean curvature of the
Simpson connection. This is defined as usual: if we consider
wedging by the K\"ahler 2-form $\omega$ as a morphism $\mathcal A^p\to
\mathcal A^{p+2}$ (where $\mathcal A^p$ is the sheaf of $\C$-valued smooth $p$-forms on $X$), and denote by $\Lambda: \mathcal A^p\to
\mathcal A^{p-2}$ its adjoint, then $\mathcal{K}_{(\fE,h)}= i
\Lambda \ccR_{(\fE,h)}$. We shall regard the mean curvature as a bilinear form on sections
of $E$ by letting
$$\mathcal{K}_{(\fE,h)}(s,t)= h(\mathcal{K}_{(\fE,h)}(s),t)\,.$$

Let us recall, for comparison and   further use, the form that the
Hitchin-Kobayashi correspondence acquires for Higgs bundles
\cite[Thm.~1]{S92}.

\begin{thm} A   Higgs vector
bundle $\fE =(E, \phi)$ over  a compact  K{\"a}hler manifold  is
polystable if and only if it admits an Hermitian metric $h$ such
that the mean curvature of  the Hitchin-Simpson connection of $(\fE,h)$
satisfies the  Hermitian-Yang-Mills condition
$$ \mathcal{K}_{(\fE,h)}= c \cdot \Id _E $$
for some constant real number $c$.
\end{thm}
The constant $c$ is related to the slope of $E$:
\begin{equation}\label{c}
c=\frac{2n\pi}{n! \vol(X)} \mu(E)
\end{equation}
where $n=\dim(X)$ and
$$\vol(X) = \frac{1}{n!}\int_X \omega^n\,.$$

\subsection{Approximate Hermitian-Yang-Mills structures and semistability}
As we already noted, and in analogy with the case of ordinary bundles, the semistability of a Higgs
bundle may be related to the existence of an approximate
Hermitian-Yang-Mills structure, which is introduced by replacing
the Chern connection by the Hitchin-Simpson connection.

\begin{defin} We say that a Higgs   bundle $\fE=(E, \phi)$
has an \emph{approximate Her\-mi\-tian-Yang-Mills struc\-tu\-re}
if for every positive real number $\xi$ there is an Hermitian
metric $h_{\xi}$ on $E$ such that
\begin{equation}
\vert \mathcal{K}_{(\fE,h)}-c \cdot \Id _E\vert  < \xi \,.
\end{equation}
\end{defin}
Here $\vert\,\vert$ is the norm defined on a self-adjoint endomorphism of $E$ as
$$ \vert \psi \vert^2  = \operatorname{max}_{X} \operatorname{tr} \psi^2\,.$$
 The
constant $c$ is again given by equation \eqref{c}.

The following result was proved in \cite{BG2}.

\begin{thm} \label{ss} \label{approHYM} A Higgs   bundle $\fE=(E, \phi)$
on a compact K\"ahler manifold
admitting an approximate Her\-mi\-tian-Yang-Mills struc\-tu\-re is
semistable.
\end{thm}

The converse implication was proved in \cite{Cardona-I} when $\dim(X)=1$, and in \cite{LiXi12} for arbitrary dimension. More exactly,

\begin{thm} \cite{LiXi12} \label{LiXi}
If $\fE=(E, \phi)$ is a semistable Higgs bundle on a compact K\"ahler manifold, then it admits an approximate Hermitian-Yang-Mills structure.
\end{thm}

\subsection{Evolution equation and Donaldson functional} A way of proving Theorem  \ref{LiXi} is to use a version of the so-called Donaldson functional for Higgs bundles. This was introduced by Simpson in \cite{S88} and was further studied in \cite{Cardona-I,Cardona-II,LiXi12}.
Let us denote by $\mathcal H^+(E)$ the space of Hermitian metrics on the vector bundle $E$. 
This is a Hilbert manifold modelled on the space  $\mathcal H(E)$ of Hermitian endomorphism of $E$ suitably completed to an $L^2$ Hilbert space.
After fixing a reference metric $k$ in $\mathcal H^+(E)$ one defines a functional on $\mathcal H^+(E)$
by letting
$$\L(h,k)=\int_X \left[ \left(i \int_0^1 tr(h_t{^{-1}}\partial h_t \cdot \mathcal R_t)dt \right)- \frac{c}{n} \log(\det(k^{-1}h)\,\omega \right]\wedge
\frac{ \omega^{n-1}}{(n-1)!}$$
where $c$ is the constant   in Equation \eqref{c}, and $h_t$ is any smooth curve in $\mathcal H^+(E)$ joining $h$ to $k$; the value of the functional turns out to be independent from the choice of this curve. 
Note that here $ \mathcal R_t$ is the curvature of the Hitchin-Simpson connection associated with the metric $h_t$ and the Higgs field $\phi$.
The flow of this functional (the ``Donaldson heat flow'')
$$h_t^{-1}\partial_t h_t =-\nabla \L =  -(\mathcal K_t -c \, \operatorname{Id}_E)$$
defines a parabolic partial differential equation which for any initial datum $h_0$ in $\mathcal H^+(E)$  has a unique smooth solution $h_t$ defined for all $t\ge 0$.

Collecting results from \cite{S88,LiXi12} one has:
 
\begin{thm} \label{Dona}
Let $\fE$ be a  Higgs vector bundle over a compact  K\"ahler manifold $X$ with K\"ahler form $\omega$,
and let $h_t$ be the solution  of the Donaldson heat flow with reference metric $k$ and initial condition $h_0$. 
Then:
\begin{enumerate} \item  the function $\L(h_t,k)$ is a monotone decreasing function of $t$;
\item if $\L(h_t,k)$ is bounded from below, then $\lim_{t\to+\infty} \vert \mathcal K_t - cI_t\vert = 0$;
\item if $\fE$  is semistable, 
$\L(h_t,k)$  is bounded from below.\end{enumerate}\end{thm}

This Theorem implies Theorem \ref{LiXi}. Moreover, if $\fE$ is stable, this can be used to prove that $\fE$ admits an Hermitian-Yang-Mills metric.

\subsection{Semistable principal bundles}\label{Sec:2}
We recall some basics about principal
bundles, notably a definition of (semi)stability for principal bundles (basic references about this topic are \cite{Rama,BalSes}).
\begin{defin}
Let $X$ be a compact  K\"ahler manifold and $G$ a reductive algebraic group over $\C$. A principal $G$-bundle  $E$ on $X$ is a complex manifold with a free action of $G$ such that $E/G \simeq X$ and  the projection $\pi:E \to X$ is locally isotrivial (locally trivial up to an \'etale cover).
\end{defin}

If $\rho\colon G \to \operatorname{Aut}(Y)$ is a
representation of $G$ as automorphisms of a variety $Y$, we may
construct the associated bundle $E(\rho)=E\times_\rho Y$, the
quotient of $E\times Y$ under the   action of $G$ given by
$(u,y)\mapsto (ug,\rho(g^{-1})y)$ for $g\in G$. If $Y=\g$ is the
Lie algebra of $G$, and $\rho$ is the adjoint action of $G$ on
$\g$, one gets the adjoint bundle of $E$, denoted by  $\Ad E $.
Another important example is obtained when $\rho$ is given by a
group homomorphism $\lambda\colon G\to G'$; in this case the
associated bundle $E'=E\times_\lambda G'$ is a principal
$G'$-bundle. We say that the structure group $G$ of $E$ has been
extended to $G'$.

If $E$ is a principal $G$-bundle on $X$, and $F$ a principal
$G'$-bundle on $X$, a morphism $E\to F$ is a pair $(f,f')$, where
$f'\colon G\to G'$ is a group homomorphism, and $f\colon E\to F$
is a morphism of bundles on $X$ which is $f'$-equivariant, i.e.,
$f(ug)=f(u)f'(g)$. Note that this induces a vector bundle morphism
$\tilde f\colon \Ad E \to\Ad(F)$ given by $\tilde
f(u,\alpha)=(f(u),f'_\ast(\alpha))$, where $f'_\ast\colon\g\to\g'$
is the morphism induced on the Lie algebras.
As an example, consider a principal $G$-bundle $E$, a group homomorphism $\lambda\colon G\to G'$,
and the extended bundle $E'$. There is a natural morphism $f\colon E\ \to E'$ defined as 
$f=\operatorname{id}\times\lambda$ if we identify $E$ with $E\times_GG$.

If $K$ is a closed subgroup of $G$, a \emph{reduction} of the structure
group $G$ of $E$ to $K$ is a principal $K$-bundle $F$ over $X$ together
with an injective $K$-equivariant bundle morphism $F \to E$.
Let $E(G/K)$ denote
the bundle over $X$ with standard fibre $G/K$ associated to $E$ via
the natural action of $G$ on the homogeneous space $G/K$.
There is an isomorphism   $E(G/K)\simeq E/K$ of bundles over $X$. Moreover,
the reductions of the structure group of $E$ to $K$ are in a one-to-one correspondence
with sections $\sigma\colon X \to E(G/K)\simeq E/K$.

 We first recall the definition of
semistable principal bundle when the base variety $X$ is a curve.
Let $T_{E/K,X}$ be the vertical tangent bundle to the bundle
$\pi_K\colon E/K \to X$.

\begin{defin} Let $E$ be a principal $G$-bundle on  a compact Riemann surface $X$. We say that
$E$ is stable (semistable) if for every proper parabolic subgroup $P\subset G$, and every reduction $\sigma\colon X \to E/P$,
the pullback $\sigma^\ast (T_{E/P,X})$ has positive (nonnegative) degree.
\end{defin}

When $X$ is a higher dimensional variety, the definition must be somewhat
refined; the introduction of an open dense subset whose complement
has codimension at least two should be compared with the definition
of (semi)stable vector bundle, which involves non-locally free subsheaves
(which are subbundles exactly on open subsets of that kind).

\begin{defin} Let $X$ be a   compact  K\"ahler manifold. A principal  $G$-bundle $E$  on $X$ is stable (semistable)
 if and only   if  for any proper parabolic subgroup $P\subset G$, any open dense subset $U\subset X$ such that $\operatorname{codim}(X-U)\ge 2$, and any  reduction $\sigma\colon U \to (E/P)_{\vert U}$ of $G$ to $P$ on $U$, one has $\deg \sigma^\ast (T_{E/P,X}) > 0$ ($\deg \sigma^\ast (T_{E/P,X}) \ge 0$).
 \label{higher}
\end{defin}

One should note that a 
line bundle defined on an open dense subset of $X$, whose complement has codimension 2 at least,  extends uniquely to the whole of $X$, so that we may consistently consider
its degree. This is discussed in detail in \cite{RamaSubra}, see also \cite{koba}, Chapter V.

\subsection{Principal Higgs bundles}\label{phb}

We switch now to principal Higgs bundles.
Let $X$ be a compact  K\"ahler manifold, and $G$ a reductive
complex  algebraic group.

\begin{defin} A principal Higgs $G$-bundle $\fE$ is a pair $(E,\phi)$, where
$E$ is a principal $G$-bundle, and $\phi$  is a global holomorphic section
 of $\Ad E \otimes\Omega^1_X$ such that $[\phi,\phi]=0$ in $\Ad E\otimes\Omega_X^2$.
\end{defin}

When $G$ is the general linear group, under the identification $\Ad E \simeq
\End(V)$, where $V$ is the vector bundle corresponding to $E$, this agrees with the usual definition  of Higgs vector bundle.

A morphism between two principal Higgs bundles $\fE=(E,\phi)$ and $\fE'=(E',\phi')$
is a principal bundle morphism $f\colon E\to E'$ such that $(f_\ast\times \operatorname{id})(\phi)=\phi'$, where
$f_\ast\colon \Ad E \to\Ad(E')$ is the induced morphism between the adjoint bundles.

Let $L$ be a closed subgroup of $G$, and $\sigma\colon X \to
E(G/L)\simeq E/L$ a reduction of the structure group of $E$ to
$L$. One has a principal $L$-bundle $F_\sigma$ on $X$ and a
principal bundle morphism $i_\sigma\colon F_\sigma\to E$ inducing
an injective morphism of bundles $\Ad(F_\sigma) \to \Ad E $. Let
$\Pi_\sigma\colon  \Ad E \otimes
\Omega^1_X\to(\Ad E /\Ad(F_\sigma))\otimes \Omega^1_X $ be the
induced projection.

\begin{defin} A section $\sigma\colon X\to E/L $ is a {\em Higgs reduction} of $(E,\phi)$
if $\phi\in\ker \Pi_\sigma$.
\end{defin}

When this happens, the reduced bundle $F_\sigma$ is equipped with a Higgs field $\phi_\sigma$ compatible with $\phi$ (i.e., $(F_\sigma,\phi_\sigma)\to (E,\phi)$ is a morphism of principal Higgs bundles).

\begin{remark} Let us again consider the case when $G$ is the general linear group $Gl(n,\C)$, and let us assume that
$P$ is a (parabolic) subgroup such that $G/P$ is the Grassmann variety $\operatorname{Gr}_k(\C^n)$ of $k$-dimensional quotients
of $\C^n$. If $V$ is the vector bundle corresponding to $E$, a reduction $\sigma$ of $G$ to $P$ corresponds to a rank $n-k$
subbundle $W$ of $V$, and the fact that $\sigma$ is a Higgs reduction means that $W$ is $\phi$-invariant, i.e., $\phi(W)\subset W\otimes \Omega^1_X$.
\end{remark}

We introduce a notion of semistability for
principal Higgs bundles (which is equivalent to the one given in
Definition 4.6 in \cite{AnBis}). We give it here when $X$ is a curve; this is generalized to
higher dimensional varieties as one does for principal bundles, see Definition \ref{higher}.

\begin{defin}\label{Hstabcurv} Let $X$ be a compact Riemann surface. A principal Higgs $G$-bundle
$\fE=(E,\phi)$ is stable (resp.~semistable)  if for every proper parabolic subgroup $P\subset G$
and every Higgs reduction $\sigma\colon X\to E/P$ one has
$\deg \sigma^\ast (T_{E/P,X})> 0$ (resp.~$\deg \sigma^\ast (T_{E/P,X})\ge 0$). \end{defin}

\bigskip\section{Hermitian-Yang-Mills structures for principal Higgs bundles}\label{Sec:3}
We  review some facts about connections on principal bundles,
covering also the case when a Higgs field is present.  Let $X$ be an $n$-dimensional   K\"ahler manifold, $G$ a complex reductive linear algebraic group and $\pi: \fE \to X$  a principal Higgs $G$-bundle.
Let $K$ be a maximal compact subgroup of $G$. Note that the Lie algebra
$\g$ of $G$ admits an involution $\iota$, called the \emph{Cartan involution}, whose +1 eigenspace is isomorphic to the Lie algebra $\mathfrak k$ of $K$. If $\fE=(E,\phi)$ is a principal Higgs $G$-bundle,
we may extend $\iota$ to an involution on the sections of the bundle $\Ad E \otimes\mathcal A^1$
(where $\mathcal A^1$ is the bundle of complex-valued smooth differential  1-forms) by letting
$$\iota(s\otimes\eta) =- \iota(s)\otimes\bar\eta\,.$$

Given  a reduction $\sigma$ of the structure group of $E$ to $K$, there is a unique connection
$D_{\sigma}$ on $E$ which is compatible with the complex structure of $E$ and with the reduction \cite{KobaNomi2}. By analogy with the vector bundle case, we call it the \emph{Chern} connection associated with the reduction $\sigma$. The Higgs field may be used to introduce another connection $$D_{\sigma,\phi} = D_\sigma + \phi + \iota(\phi)$$
which we call the \emph{Hitchin-Simpson connection} of the triple $(\fE,\sigma)=(E,\phi,\sigma)$. 

To state our results we need the notion of \emph{polystable} principal Higgs bundle.
We may associate with every character $\chi$ of $K$
a line bundle $L_\chi=E\times_\chi \C$ on $E/K$, where we regard $E$ as a principal
$K$-bundle on $E/K$. We call $\mu_\sigma$, the slope of a   reduction $\sigma$,
the group homomorphism $\mu_\sigma\colon \mathcal X(K) \to \Q$ (where $ \mathcal X(K) $
is the group of characters of $K$) which to any character $\chi$ associates the degree
of the line bundle $\sigma^\ast(L_\chi^\ast)$ \cite{Rama}.

\begin{defin} \label{admi}
A reduction $\sigma$ of the structure group of $G$ of a principal Higgs $G$-bundle
$\fE$ to a parabolic subgroup $P\subset G$  is said to be \emph{admissible} if $\mu_\sigma(\chi)=0$ for
every character of $\chi$ of $P$ which vanishes on the centre of $G$.
\end{defin}
\begin{defin} A principal Higgs $G$-bundle $\fE$ is said to be \emph{polystable}
if there is a parabolic subgroup $P$ of $G$ and a Higgs reduction $\sigma$ of the structure
group of $E$ to a Levi subgroup $L$ of $P$ such that
\begin{enumerate} \item
the reduced principal Higgs $L$-bundle $\fE_\sigma$ is stable;
\item the   principal Higgs $P$-bundle obtained by extending the structure
group of $\fE_\sigma$ to $P$ is an admissible reduction of the structure group of $\fE$ to
$P$ (cf.~Definition \ref{admi}). \end{enumerate}\end{defin}

Also in this case one has a \emph{Hitchin-Kobayashi correspondence} \cite{AnBis}. We
say that a reduction $\sigma$ of the structure group $G$ of a principal Higgs $G$-bundle $\fE$ to a maximal compact  subgroup $K$ is \emph{Hermitian-Yang-Mills}  if there is an element $\tau$ in the centre
$\mathfrak z$ of the Lie algebra $\g$ of $G$ such that
$$ \mathcal K_{\sigma,\phi} = \tau$$
where $ \mathcal K_{\sigma,\phi} $ is the \emph{mean curvature} of the Hitchin-Simpson connection.\footnote{Note that if  $\tau$ is an element in the centre $\mathfrak z$ of $\mathfrak g$, the constant equivariant  morphism $f_\tau\colon E \to \mathfrak g$, $f_\tau(u) =\tau$, defines a section of $\Ad E$.}
\begin{thm} \cite{BisSchu} \label{HK} A principal Higgs $G$-bundle $\fE$ is polystable if and only
if it admits an Hermitian-Yang-Mills reduction to a maximal compact subgroup $K\subset G$.
\end{thm}

This notion of polystability extends the one holding for Higgs vector bundles,
i.e., a Higgs vector bundle is polystable if it is a direct sum of stable Higgs vector bundles having the
same slope. 

The proof of the following
result is implicitly contained in \cite{AnBis}.

 \begin{prop}
A principal Higgs bundle is polystable if and only if its adjoint bundle is polystable.
\label{polyrem}
\end{prop}

\begin{defin}
A   principal Higgs $G$-bundle $\fE$  over $X$ is said to admit an \emph{approximate Hermitian-Yang-Mills structure} if there exist an element $\tau \in \z$, where $\z$ is the center of $\g$, and for every $\xi > 0$, a   reduction $\sigma_{\xi}$ of $G$ to $K$, such that $$|\mathcal K_{\sigma_{\xi},\phi}-\tau|< \xi.$$
\end{defin}
Here the norm $\vert\,\cdot\,\vert$ on the sections of $\Ad E$ is defined as
$$\vert\psi\vert= \operatorname{max}_{x\in X} \sqrt{\kappa(\psi,\iota(\psi))}$$
where $\kappa$ is an ad-invariant nondegenerate bilinear form on $\mathfrak g$ which extends the Killing-Cartan form on the semisimple part of $\mathfrak g$.
\bigskip

\section{Existence of approximate Hermitian-Yang-Mills structures}\label{Sec:3}
In this Section we prove the following result.

\bigskip

\begin{thm}\label{MainThm}
A   principal Higgs $G$-bundle $\fE$ over $X$ is semistable if and only if it admits an approximate Hermitian-Yangs-Mills structure.
\end{thm}

We start with some preliminary constructions. To deal with the approximate Hitchin-Kobayashi correspondence for $\fE$,
it is convenient to associate a metric to every reduction $\sigma$ of the structure group $G$ of $\fE$ to a maximal compact subgroup $K$. To this end we consider the adjoint representation $\Ad\colon G \to Gl(\mathfrak g)$ and choose a maximal compact subgroup $\tilde K$ of $Gl(\mathfrak g)$  containing $\Ad(K)$. Note that the principal Higgs bundle obtained by extending the structure group of $\fE$ to $Gl(\mathfrak g)$ via the morphism $\Ad$ is isomorphic to the Higgs bundle of linear frames $L(\Ad \fE)$ of the adjoint Higgs bundle $(\Ad E, \tilde{\phi})$, where $\tilde{\phi}$ is the Higgs map on $\Ad E $ via the morphism $(\Ad \times id)_\ast$.
The reduction $\sigma$ induces an hermitian metric $h_\sigma$ in $\Ad E$.
The adjoint morphism $\Ad$ also induces a morphism $\Ad\colon \Ad E_\sigma\to \End(\Ad E)$, and therefore also a morphism
$\Ad \colon\Omega^2( \Ad E_\sigma)\to \Omega^2(\End(\Ad E))$. It is easy to check that this morphism identifies the curvature of the Chern connection $D_\sigma$ with the curvature of the Chern connection of the metric $h_\sigma$, which are elements in $\Omega^2( \Ad E_\sigma)$ and $\Omega^2(  \End(\Ad E))$, respectively. So the curvature of the Hitchin-Simpson connection $D_{\sigma, \phi}$ on $\fE$ goes through the $\Ad$ morphism to the curvature of the Hitchin-Simpson connection $\D_{(\sigma, h_\sigma)}$  associated to the Hermitian metric $h_\sigma$ on $\Ad E$.

 \begin{prop} Assume that a metric $h$ on $\Ad E$ has a connected holonomy group. Then $h$ comes from a   reduction of $G$ to $K$ according to the previous construction if and only if the curvature of its
Chern connection lies in the image of $\Ad$.
 \label{p1}
 \end{prop}
 \begin{proof} 
Let $H(h)$ be the holonomy group of the Chern connection of  $h$. After fixing a point $u$ in the principal   bundle $L(\Ad E)$, the latter has a reduced   bundle $P(h,u)$ with structure group $H(h)$.  
The curvature of the Chern connection acting on the tangent vectors to $L(\Ad E)$ generates the Lie algebra of the holonomy group  $H(h)$. Under our assumptions, this implies that the Lie algebra of  $H(h)$  is contained in the Lie algebra of $K$ (more precisely, it is contained in the image of the Lie algebra of $K$ under the adjoint map).  Since  the   holonomy group  $H(h)$ is connected,  it  is contained in $K$.

This implies that the holonomy bundle $P(h,u)$ can be regarded as a   reduction $E'$ of the structure group of $E$ to  $H(h)$.
Moreover, $E'$ induces  a reduced  subbundle $E''\subset E$ with structure group $K$.\footnote{This is a general feature: if $E\to X$ is a principal Higgs $G$-bundle, $G'\subset G''\subset G$ are nested subgroups, and $\fE'$ is a reduction of the structure group of $E$ to $G'$, the latter corresponds to a section $X \to E/G'$; composing with the morphism $E/G'\to E/G''$, this gives a reduction $E''$ of $E$ to $G''$. Note that
$E'' \simeq E' \times_{G'}G''$.} By construction the corresponding   reduction of $E$ to $K$  induces in $\Ad E$ the metric $h$. 
 \end{proof}
 
 \begin{remark} Note that if $\fE=(E,\phi)$ is a Higgs bundle with a metric $h$, 
the  curvature of the
 Hitchin-Simpson connection $\mathcal D_{h,\phi}$  lies in the image of $\Ad$ if and only if the curvature of the Chern connection $D_h$ does.  \end{remark}

 \begin{prop} {\rm \cite{Nom56}} There exist reductions of $E$ to $K$ whose holonomy group is $K$.
 \label{goodhol}
 \end{prop}

Note that since $K$ is a maximal compact subgroup of a connected reductive group over $\C$, it is connected.

\begin{lemma} Let $h_t$ the solution of the Donaldson flow with initial datum $h_0$. If the holonomy of $h_0$ is connected, the same is true for the holonomy group of $h_t$ for every $t \ge 0$.
\label{connected}
\end{lemma}  

\begin{proof}
Realize $H(h_0)$ as a subgroup of $Gl(\g)$ by choosing a point $u\in L(\Ad E)$.
Define a map $f\colon H(h_0) \times \R_{\ge 0} \to Gl(\g)$ as follows.   For $(g,t)\in H(h_0) \times \R_{\ge 0}$ represent $g$ by a loop $\gamma\colon [0,1]\to X $ based at the base point of the fibre containing $u$;
consider the curve $\tilde\gamma$ obtained in $L(\Ad E)$ by horizontally lifting $\gamma$ starting from $u$ using the Chern connection of the metric $h_t$. Let $g_t\in Gl(\g)$ be such that $\tilde\gamma(t) = u\,g_t$. Then $f(g,t)=g_t$. This is independent of the choice of $\gamma$ and is a continuous map, hence its image for fixed $t$ is connected. But the latter is exactly the holonomy group $H(h_t)$.
\end{proof}

Let    $\mathcal H^+(\Ad E)$ be the space of Hermitian metrics on $\Ad E$. Let $\mathcal H^+_c(\Ad E)$ be the subspace of $\mathcal H^+(\Ad E)$ formed by the metrics whose holonomy group is connected. In view of Lemma \ref{connected}, one can restrict to this space without loss of generality. We also note that $\mathcal H^+_c(\Ad E)$ is an open submanifold of $\mathcal H^+(\Ad E)$. Finally,
let  $\mathcal H^+_K(\Ad E)$ be the subspace of  $\mathcal H^+_c(\Ad E)$ formed by the metrics that come from   reduction of $\fE$ to $K$.

\begin{lemma} \label{sub} $\mathcal H^+_K(\Ad E)$ is a closed submanifold of  $\mathcal H^+_c(\Ad E)$.\end{lemma}

\begin{proof} We give only a sketch of the proof. By using an Ad-invariant bilinear form on the Lie algebra $\mathfrak g$ (the Killing-Cartan form on the semisimple part of $\mathfrak g$ and any nondegenerate  bilinear form on the radical of $\mathfrak g$), the K\"ahler metric on $X$, and integration on $X$, one can endow the vector space $\Omega^2(\End(\Ad E))$ with a scalar 
product;\footnote{Actually this space should be suitably completed using Sobolev norms, but we shall omit these details here.} then we can orthogonally split this space into the image of the adjoint map Ad and its complement $\mathrm{Ad}^\perp$. We denote by $r:\Omega^2(\End(\Ad E)) \to \mathrm{Ad}^\perp$ the projection and
 introduce a map
$$ \varpi \colon \mathcal H^+_c(\Ad E) \to \mathrm{Ad}^\perp$$
which to a metric $h$ on $E$, having connected holonomy group, associates the projection $r(\mathcal R_{h,\phi})$
onto $ \mathrm{Ad}^\perp$ of the curvature of the Hitchin-Simpson connection $\mathcal D_{h,\phi}$. By Proposition \ref{p1}, $\mathcal H^+_K(\Ad E)$ coincides with the locus $\varpi^{-1}(0)$. By studying this map one proves the claim.  \end{proof}

 \begin{thm} \label{remains} There exist  metrics $h_0$ on $\Ad E$, coming from a   reduction of $\fE$ to $K$, such that, denoting by $h_t$ the solution of the Donaldson flow with initial datum $h_0$, for every $t\ge 0$ the metric $h_t$ comes from a   reduction of $\fE$ to $K$.\label{preserve}
 \end{thm}

 \begin{proof} We take an $h_0$ induced by a    reduction $\sigma$ of $\fE$ to $K$ whose holonomy group is $K$ (cf.~Prop.~\ref{goodhol}). 
 If  the gradient vector field of the Donaldson functional is tangent  to the submanifold $\mathcal H^+_K(\Ad E)$,  the Donaldson flow remains inside the latter. So we need to prove that claim. 
 Under our assumptions, the gradient flow of the Donaldson functional at a point $h \in \mathcal H^+(\Ad E)$ is proportional to the mean curvature $\mathcal K_{h,\phi}$ of the Hitchin-Simpson connection $\mathcal D_{h,\phi}$, which takes values in the image of Ad.  On the other hand, the normal vector field to the submanifold $\mathcal H^+_K(\Ad E)$ is orthogonal to the image of Ad for what we saw in Lemma \ref{sub}.
 This   implies that the gradient flow is contained in the submanifold $\mathcal H^+_K(\Ad E)$.
 \end{proof}

We can now prove Theorem \ref{MainThm}. Following the argument in  \cite{BisJacStem} we can reduce to the case of semisimple structure group. If $G$ is semisimple,  the center $\z$ of the Lie algebra $\g$ is trivial and the formula in the thesis becomes $$|\mathcal K_{\sigma_\xi,\phi}|< \xi\,.$$

If the principal Higgs $G$-bundles $\fE$ admits an approximate Hermitian-Yang-Mills structure, then the adjoint bundle $\Ad(\fE)$ also does, as we noted above. By Theorem \ref{approHYM} on \cite{BG2} we conclude that $\Ad(\fE)$ is semistable. But this holds true if and only if $\fE$ is semistable.

For the converse implication, assume that $\fE$ is semistable, so that $\Ad(\fE)$ is semistable as well. We consider the Donaldson flow for the Higgs bundle $\Ad \fE$ with initial metric $h_0$ that comes from a reduction of $\fE$ to $K$ and has connected holonomy group. By Theorem \ref{remains}, the ensuing flow remains inside the submanifold $\mathcal H^+_K(\Ad E)$. In view of  Theorem \ref{Dona}, this induces an approximate Hermitian-Yang-Mills structure on $\Ad \fE$ which can be regarded as approximate  Hermitian-Yang-Mills structure on $\fE$.

\end{document}